\documentclass[10pt,reqno]{amsart}
\usepackage{amsmath,amsthm,amssymb}
\usepackage{hyperref}
\usepackage[mathscr]{eucal}
\usepackage{stmaryrd}  

\usepackage{tikz} 
\usepackage{subfig} 
\input xy
\xyoption{all}
\usepackage{enumerate}
\numberwithin{equation}{section}
\theoremstyle{definition}
\newtheorem{dfn}{Definition}[section]
\newtheorem{example}[dfn]{Example}
\newtheorem{rem}[dfn]{Remark}
\theoremstyle{plain}
\newtheorem{thm}[dfn]{Theorem}

\newtheorem{cor}[dfn]{Corollary}
\newtheorem{lem}[dfn]{Lemma}

\title{De Concini-Kac filtration and Gelfand-Tsetlin generators for quantum $\mathfrak{gl}_N$}

\author{Vyacheslav Futorny}
\author{Jonas T. Hartwig }
\address{Department of Mathematics,
 University of S\~ao Paulo,
 S\~ao Paulo, Brazil and Max Planck Institute for Mathematics, Bonn, Germany}
 \email{futorny@ime.usp.br}
\address{Department of Mathematics, Iowa State University, Ames, IA-50011, USA}
\email{jth@iastate.edu}

\newcommand\al{\alpha}\newcommand\be{\beta} 
\newcommand\ep{\varepsilon}
   
 \newcommand\la{\lambda}\newcommand\La{\Lambda}
  \newcommand\si{\sigma}

\newcommand{\iv}[2]{\llbracket #1,#2 \rrbracket}  

\newcommand\mf{\mathfrak}

\newcommand\C{\mathbb{C}}    \newcommand\Z{\mathbb{Z}}

\newcommand{\gl}{\mathfrak{gl}}

\newcommand{\lt}{\mathrm{lt}}
\newcommand{\height}{\mathrm{ht}}

\DeclareMathOperator{\gr}{gr}
  \DeclareMathOperator{\diag}{diag}

\DeclareMathOperator{\Frac}{Frac}

\begin{document}

\begin{abstract}
In this note we compute the leading term with respect to the De Concini-Kac filtration of $U_q(\mathfrak{gl}_n)$ of a generating set for the quantum Gelfand-Tsetlin subalgebra.
\end{abstract}

\maketitle

\section{Introduction}

An important class of associative algebras, called {\em Galois rings}
was introduced in \cite{FO1}. This class of algebras
includes for example Generalized Weyl algebras over integral
domains with infinite order automorphisms (in particular, the $n$-th Weyl
algebra, the quantum plane, $q$-deformed Heisenberg algebra,
quantized Weyl algebras, Witten-Wo\-ro\-no\-wicz algebra
\cite{ba}, \cite{bavo}; the universal enveloping algebra of
$\gl_N$ over the Gelfand-Tsetlin subalgebra \cite{dfo:gz},
\cite{dfo:hc}, associated shifted Yangians and finite
$W$-algebras \cite{FMO}, \cite{FMO1}.

These algebras contain a special commutative subalgebra $\Gamma$ which
allows one to embed the algebra into a certain invariant subalgebra
of some skew group algebra. In particular, such an embedding enables
the computation of the skew field of fractions \cite{FMO},\cite{FH}.

A natural choice of a commutative subalgebra in many associative
algebras is a so-called Gelfand-Tsetlin subalgebra. Classical
Gelfand-Tsetlin subalgebras of the universal enveloping algebras
of a simple Lie algebras were considered in
 \cite{FM}, \cite{Vi}, \cite{KW-1}, \cite{KW-2}, \cite{G1}, \cite{G2}
among the others.

In this paper we study the quantized enveloping algebra $U_q(\mathfrak{gl}_N)$.
This algebra contains a quantum analog of the Gelfand-Tsetlin
subalgebra of $U(\mathfrak{gl}_N)$, which we denote by $\Gamma_q$.
Based on the properties of so called generic Gelfand-Tsetlin modules constructed
in \cite{MT}, it was shown in \cite{FH} that $U_q(\mathfrak{gl}_N)$ is a Galois ring with
respect to $\Gamma_q$. This allowed us to prove the quantum Gelfand-Kirillov conjecture for $U_q(\mathfrak{gl}_N)$ \cite{FH},\cite{F}. Unlike all the examples listed above, $U_q(\mathfrak{gl}_N)$ is a Galois rings
with respect to a subalgebra which not a polynomial algebra.

Our main result is the calculation of the leading terms of a set of generators $d_{rs}$ for the quantum Gelfand-Tsetlin subalgebra.

\begin{thm}\label{thm:leadingTerm-intro}
The leading term of $d_{rs}$ (see \eqref{eq:drs}), with respect to the De Concini-Kac filtration using \eqref{eq:w0decomp} as decomposition of the longest Weyl group element, is 
\begin{equation}\label{eq:ltdrs-intro}
\lt(d_{rs}) = \la\cdot
 t_{1+s,1}^{(0)}t_{2+s,2}^{(0)}\cdots t_{r,r-s}^{(0)}\cdot t_{1,r-s+1}^{(1)}t_{2,r-s+2}^{(1)}\cdots t_{s,r}^{(1)}
\end{equation}
for some nonzero $\la\in\C$. 
\end{thm}

\subsection*{Notation}
$\iv{a}{b}$ denotes the set $\{x\in \Z\mid a\le x\le b\}$. The cardinality of a set $S$ is denoted $\# S$.
Throughout this paper, the ground field is $\C$ and $q\in\C$ is nonzero and
not a root of unity. We put $\C^\times=\C\setminus\{0\}$.

\subsection*{Acknowledgment}
The authors are grateful to A.Molev for helpful comments. The
first author is grateful  to  Max Planck Institute in Bonn for
support and hospitality during his visit. The first author is
supported in part by the CNPq grant (301743/2007-0) and by the
Fapesp grant (2010/50347-9).

\section{The algebra \texorpdfstring{$U_q(\mathfrak{gl}_N)$}{Uq(glN)}}
In this section we recall some facts about the quantized enveloping algebra
$U_q(\mathfrak{gl}_N)$ which will be used.

\subsection{Definition}
For positive integers $N$ we let $U_N=U_q(\mathfrak{gl}_N)$ denote the
unital associative $\C$-algebra with generators $E_i^\pm$, $K_j,
K_j^{-1}$, $i\in\iv{1}{N-1}$, $j\in\iv{1}{N}$ and relations
\cite[p.163]{KS}
\begin{gather*}
K_iK_i^{-1}=K_i^{-1}K_i=1, \quad [K_i,K_j]=0,\quad\forall i,j\in\iv{1}{N},\\
\begin{aligned}
K_iE_j^\pm K_i^{-1} &= q^{\pm(\delta_{ij}-\delta_{i,j+1})}E_j^{\pm}, \quad\forall i\in\iv{1}{N}, \forall j\in\iv{1}{N-1},\\
[E_i^+,E_j^-]&=\delta_{ij}\frac{K_iK_{i+1}^{-1}-K_{i+1}K_i^{-1}}{q-q^{-1}}, \quad\forall i,j\in\iv{1}{N-1},\\
[E_i^{\pm},E_j^{\pm}]&=0,\quad |i-j|>1,
\end{aligned}\\
(E_i^\pm)^2E_j^\pm -(q+q^{-1})E_i^\pm E_j^\pm E_i^\pm + E_j^\pm (E_i^\pm)^2 =0, \quad |i-j|=1.
\end{gather*}

\subsection{De Concini-Kac filtration} \label{sec:filt}
\cite[Section I.6.11]{BG}
Let $\al_i=\ep_i-\ep_{i+1}$, $i\in\iv{1}{N-1}$ be the standard simple roots of $\mathfrak{gl}_N$ where $\ep_i(\diag(a_1,\ldots,a_N))=a_i$.
Fix the following decomposition of the longest Weyl group element:
\begin{equation}\label{eq:w0decomp}
w_0=s_{i_1}\cdots s_{i_M} = (s_1s_2\cdots s_{N-1})(s_1s_2\cdots s_{N-2})
\cdots (s_1s_2)s_1,
\end{equation}
where $s_i=(i\; i+1)\in S_N$, and $M=N(N-1)/2$.
Let $\{\be_j=s_{i_1}\cdots s_{i_{j-1}}(\al_{i_j})\}_{j=1}^M$ be the corresponding enumeration of positive roots of $\mathfrak{gl}_N$.
One checks that
\begin{equation}\label{eq:posRoots}
(\be_1,\be_2,\ldots,\be_M)=(\be_{12},\be_{13},\ldots,\be_{1N},\, \be_{23},\be_{24},\ldots,\be_{2N},\,\ldots ,\, \be_{N-1,N}),
\end{equation}
where $\be_{ij}=\ep_i-\ep_j$ for all $i,j\in\iv{1}{N}$, $i<j$.
Let $E_{\be_i}$, $F_{\be_i} \in U_q(\mathfrak{gl}_N)$ be the corresponding positive and negative root vectors (see e.g. \cite[Section I.6.8]{BG}).
The following PBW theorem for $U_q(\mathfrak{gl}_N)$ is well-known:
\begin{thm}
The set of ordered monomials
\begin{equation}\label{eq:PBW}
F^rK_\lambda E^k := F_{\be_1}^{r_1}\cdots F_{\be_M}^{r_M} \cdot K_1^{\la_1}\cdots K_N^{\la_N}\cdot E_{\be_1}^{k_1}\cdots E_{\be_M}^{k_M}
\end{equation}
where $r,k\in\Z_{\ge 0}^M$ and $\la\in \Z^N$,
form a basis for $U_q(\mathfrak{gl}_N)$.
\end{thm}

Define the \emph{total degree} of a monomial $F^rK_\lambda E^k$ to be
\begin{equation}\label{eq:totalDegree}
d(F^rK_\lambda E^k)= \big(k_M,\ldots,k_1,r_1,\ldots, r_M, \height(F^rK_\lambda E^k)\big)\in \Z_{\ge 0}^{2M+1},
\end{equation}
where
\begin{equation}\label{eq:heightMonDef}
\height(F^rK_\lambda E^k) = \sum_{j=1}^M (k_j+r_j)\height(\be_j)
\end{equation}
and $\height(\be)=\sum_{i=1}^{N-1} a_i$ if $\be=\sum_{i=1}^{N-1} a_i\alpha_i$.
Equip the monoid $\Z_{\ge 0}^{2M+1}$ with the lexicographical order
uniquely determined by the inequalities
\[u_1<u_2<\cdots <u_M\]
where $u_i=(0,\ldots,0,1,0,\ldots,0)$ with $1$ on the $i$:th position.

\begin{thm}[De Concini-Kac]
\label{thm:DeConciniKac}
The total degree function $d$ defined above
equips $U=U_q(\mathfrak{gl}_N)$ with a $\Z_{\ge 0}^{2M+1}$-filtration $\{U_{(k)}\}_{k\in\Z_{\ge 0}^{2M+1}}$. The
 associated graded algebra $\gr U$ is the $\C$-algebra on the generators
\[ \bar E_{\be_i}, \bar F_{\be_j}, \bar K_\al\]
$i=1,\ldots,M$, $\al\in\Z^N$ subject to the following defining relations:
\begin{equation}
\begin{aligned}
\bar K_\al \bar K_\be &= \bar K_{\al+\be} &
\bar K_0 &= 1\\
\bar K_\al \bar E_{\be_i} &= q^{(\al,\be_i)} \bar E_{\be_i}\bar K_\al  &
\bar K_\al \bar F_{\be_i} &= q^{-(\al,\be_i)}\bar F_{\be_i}\bar K_\al \\
\bar E_{\be_i} \bar F_{\be_j} &= \bar F_{\be_j} \bar E_{\be_i} &&\\
\bar E_{\be_i} \bar E_{\be_j} &= q^{(\be_i,\be_j)}\bar E_{\be_j}\bar E_{\be_i} &
\bar F_{\be_i} \bar F_{\be_j} &= q^{(\be_i,\be_j)}\bar F_{\be_j}\bar F_{\be i} \\
\end{aligned}
\end{equation}
for $\al,\be\in \Z^N$ and $1\le i,j\le M$.
\end{thm}
\begin{proof} That $d$ actually defines a filtration follows from the
commutation relation known as the \emph{Levendorski\u{\i}-Soibelman straightening rule} \cite[Proposition 5.5.2]{LS}. See \cite[Proposition 1.7]{DK} for details.
\end{proof}

Observe that the root vectors $E_\al, F_\al$, hence the De Concini-Kac filtration, depend on the choice of decomposition of the longest Weyl group element.

A simple but important corollary which will be used implicitly throughout is that
\begin{equation}
d(ab)=d(a)+d(b)=d(ba)
\end{equation}
for all $a,b\in U_q(\mf{gl}_N)$, where now $d(a)$ denotes the smallest $k\in\Z_{\ge 0}^{2M+1}$ such that $a\in U_{(k)}$.
This follows from the fact that the associated graded algebra is a domain.

\subsection{RTT presentation}
$U_q(\mathfrak{gl}_N)$ has an alternative presentation. It is isomorphic
to the algebra with generators $t_{ij}$, $\bar t_{ij}$, $i,j\in\iv{1}{N}$ and relations
\begin{subequations}\label{eq:RTT}
\begin{align}
\label{eq:RTTa}
t_{ij} &= 0 = \bar t_{ji},\quad \forall i<j,\\
\label{eq:RTTb}
t_{ii}\bar t_{ii} &= 1 = \bar t_{ii} t_{ii},\quad \forall i,\\
\label{eq:RTTc}
q^{\delta_{ij}} t_{ia}t_{jb} - q^{\delta_{ab}} t_{jb}t_{ia} &= (q-q^{-1})
(\delta_{b<a} - \delta_{i<j}) t_{ja} t_{ib} \\
\label{eq:RTTd}
q^{\delta_{ij}} \bar t_{ia} \bar t_{jb} - q^{\delta_{ab}}\bar t_{jb}\bar t_{ia} &= (q-q^{-1})
(\delta_{b<a} - \delta_{i<j}) \bar t_{ja} \bar t_{ib} \\
\label{eq:RTTe}
q^{\delta_{ij}}\bar t_{ia} t_{jb} - q^{\delta_{ab}} t_{jb} \bar t_{ia} &= (q-q^{-1})(\delta_{b<a}t_{ja}\bar t_{ib} - \delta_{i<j} \bar t_{ja} t_{ib} )
\end{align}
\end{subequations}
for all $i,a,j,b\in\iv{1}{N}$, where $\delta_S$ equals $1$ is $S$ is true and $0$ if $S$ is false.
An identification of the two sets of generators is given by \cite[Section 8.5.4]{KS}:
\begin{equation} \label{eq:identification}
\begin{aligned}
\bar t_{ii} &= K_i^{-1} & t_{ii} &= K_i \\
\bar t_{i,i+1} &= (q-q^{-1})K_i^{-1}E_i &
t_{i+1,i} &= -(q-q^{-1})F_iK_i \\
\bar t_{ij} &= (q-q^{-1})(-1)^{i-j+1}K_i^{-1} E_{\be_{ij}} &
t_{ji} &= -(q-q^{-1})F_{\be_{ij}} K_i
\end{aligned}
\end{equation}
for $j>i+1$,
where $E_{\be_{ij}}, F_{\be_{ij}}$ are the root vectors, defined in Section \ref{sec:filt}.

\subsection{Gelfand-Tsetlin subalgebra }
Let $U_q=U_q(\mf{gl}_N)$. It is immediate by the defining relations that,
for each $r\in\iv{1}{N}$, the subalgebra $U_q^{(r)}$ of $U_q$ generated by
$E_i,F_i,K_j$ for $i\in\iv{1}{r-1},\, j\in\iv{1}{r}$ (or equivalently,
by $t_{ij},\bar t_{ij}$ for $i,j\in\iv{1}{r}$) can be identified with
$U_q(\mathfrak{gl}_r)$.
Thus we have a chain of subalgebras
\[U_q^{(1)}\subset U_q^{(2)}\subset \cdots \subset U_q^{(N)}=U_q.\]
Let $Z_r$ denote the center of $U_q^{(r)}$.
The subalgebra of $U_q$ generated by $Z_1,\ldots,Z_N$ is called the
\emph{Gelfand-Tsetlin subalgebra} and will be denoted by $\Gamma_q$.
It is immediate that $\Gamma_q$ is commutative.

In \cite[Section 5]{HM} it is proved that $Z_r$
 is generated by the coefficients of the following polynomial in $U_q^{(r)}[u^{-1}]$:
\begin{equation}\label{eq:zNu}
z_r(u)=\sum_{\si\in S_r} (-q)^{-l(\si)} \prod_{j=1}^r \big(t_{\si(j)j}-\bar t_{\si(j)j}q^{2(j-1)}u^{-1}\big).
\end{equation}
It will be useful to rewrite this polynomial in a different way.
For this purpose it will be convenient to use the notation
\begin{equation}\label{eq:tijk}
t_{ij}^{(k)} = \begin{cases}t_{ij},& k=0, \\ \bar t_{ij},& k=1.\end{cases}
\end{equation}
A direct computation gives that
\begin{equation}
z_r(u)=
\sum_{s=0}^r (-1)^r
d_{rs}(q^2u)^{-s},
\end{equation}
where
\begin{equation}\label{eq:drs}
d_{rs} =
\sum_{\si\in S_r} (-q)^{-l(\si)}
\sum_{k\in\{0,1\}^r:\, \sum k_i=s}
q^{2(k_1+2k_2+\cdots+rk_r)} t_{\si(1)1}^{(k_1)}\cdots t_{\si(r)r}^{(k_r)}.
\end{equation}

Observe that $d_{r0} = d_{rr}^{-1}$. Therefore,
the (commuting) elements $d_{rs}$, $1\le s\le  r\le N$, generate $\Gamma_q$,
provided we allow taking negative powers of $d_{rr}$. We show that they are algebraically independent.

\begin{lem}\label{lem:Gammaiso}
\begin{equation}
\Gamma_q \simeq \C[d_{rs}\mid 1\le s\le r\le N][d_{rr}^{-1}\mid 1\le r\le N].
\end{equation}
\end{lem}
\begin{proof}
By applying the quantum Harish-Chandra isomorphism
$h_r:Z_r\to (U_r^{0})^{W_r}$ (see \cite[Lemma 5.3]{FH}) to the polynomial $z_r(u)$ from \eqref{eq:zNu} (as in \cite[Section 5]{HM}) we get
\begin{align*}
h_r(z_r(u)) =& (K_1-K_1^{-1}u^{-1})(K_2-q^2K_2^{-1}u^{-1})\cdots (K_r-q^{2(r-1)}K_r^{-1}u^{-1}) \\
=& q^{r(r+1)}(K_1\cdots K_r)^{-1}\prod_{j=1}^r (q^{-2j}K_j^2-(q^2u)^{-1})
\end{align*}
So
\[h_r(d_{rs}) = q^{r(r+1)/2}(\widetilde{K}_1\cdots\widetilde{K}_r)^{-1}\cdot e_{rs}(\widetilde{K}_1^2,\ldots,\widetilde{K}_r^2),\quad r\in\iv{1}{N}, s\in\iv{0}{r}\]
where $\widetilde{K}_i = q^{-i}K_i$, and $e_{rs}$ is the elementary symmetric polynomial in $r$ variables of degree $s$. By the proof of \cite[Lemma 5.3]{FH}, this shows that
\begin{equation}\label{eq:Zriso}
Z_r\simeq \C[d_{rs}\mid s=1,2,\ldots,r][d_{rr}^{-1}].
\end{equation}
Recall that $\La^G\simeq \La_1^{W_1}\otimes\cdots\otimes \La_N^{W_N}$.
As shown in \cite{FH} there is an injective map $\varphi:U\to ((\Frac \La)\ast\Z^{n(n-1)/2})^G$ such that $\varphi$ restricts to an isomorphism $\varphi|_{\Gamma_q}:\Gamma_q\to \La^G$
and $\varphi_i:=\varphi|_{Z_m}: Z_m\to \La_m^{W_m}$ for each $m\in\iv{1}{N}$.
Thus we have a commutative diagram
\[
\xymatrix@C=1.5cm@M=1ex{ 
 \Gamma_q \ar[r]^{\varphi|_{\Gamma_q}} &  \La^G\\
 Z_1\otimes\cdots\otimes Z_N \ar[u]^f \ar[r]^{\varphi_1\otimes\cdots\otimes \varphi_N} & \La_1^{W_1}\otimes\cdots\otimes\La_N^{W_N} \ar[u]_g
}
\]
where the vertical arrows are given by multiplication. The horizontal maps and $g$ are isomorphisms. Hence $f$ is an isomorphism. Combining this fact with \eqref{eq:Zriso} we obtain the required isomorphism.
\end{proof}

\section{Leading term of generators}

In this section we prove the main theorem which determines the leading term of each of the generators $d_{rs}$ of $\Gamma_q$ with respect to the De Concini-Kac filtration.

\begin{thm}\label{thm:leadingTerm}
The leading term of $d_{rs}$ (see \eqref{eq:drs}), with respect to the De Concini-Kac filtration using \eqref{eq:w0decomp} as decomposition of the longest Weyl group element, is obtained by taking
\[\si = (1\; 2\; \cdots\; r)^s.\]
in the sum \eqref{eq:drs}. That is,
\begin{equation}\label{eq:ltdrs}
\lt(d_{rs}) = \la\cdot
 t_{1+s,1}^{(0)}t_{2+s,2}^{(0)}\cdots t_{r,r-s}^{(0)}\cdot t_{1,r-s+1}^{(1)}t_{2,r-s+2}^{(1)}\cdots t_{s,r}^{(1)}
\end{equation}
for some nonzero $\la\in\C$. 
\end{thm}

\begin{example}
As an example, we determine directly the leading term of $d_{42}$. The most significant
component of the total degree \eqref{eq:totalDegree} is the height. Using \eqref{eq:heightDef}-\eqref{eq:heightCorr},
it is easy to see that there are four permutations in $S_4$ which gives the maximal possible height $8$:
\[(13)(24),\quad (14)(23),\quad (1324),\quad (1423).\]
The monomial associated to such a permutation $\si$ is
\[t_{\si(1)1}^{(k_1)}t_{\si(2)2}^{(k_2)}t_{\si(3)3}^{(k_3)}t_{\si(4)4}^{(k_4)}\]
where $k_i=0$ if $\si(i)>i$ and $k_i=1$ if $\si(i)<i$.
After the height we need to compare the exponent of $F_{\be_{34}}$ in the four different monomials,
because $\be_{34}$ is the largest positive root in the ordering
\[\be_{12} < \be_{13}<\be_{14}<\be_{23}<\be_{24}<\be_{34}\]
(see \eqref{eq:posRoots}).
This exponent is the same as the exponent (either $1$ or $0$) of $t_{43}^{(0)}$ due to the identifications \eqref{eq:identification}.
But this exponent is $0$ in all four cases because none of the permutations map $3$ to $4$.

So we look at the second largest positive root, which is $\be_{24}$. As in the previous case,
we ask if $\si(2)=4$ in any of the four permutations. There are two for which this holds,
$(13)(24)$ and $(1324)$. The others do not map $2$ to $4$ which means their corresponding
monomials are of lower total degree.

To compare the two candidates $(13)(24)$ and $(1324)$ we look at the third largest root, $\be_{23}$.
But $\si(2)\neq 3$ in both. Next is $\be_{14}$ but again $\si(1)\neq 4$ in both.
Next is $\be_{13}$ and now $\si(1)=3$ for both $\si=(13)(24)$ and $\si=(1324)$.
 Next is $\be_{12}$ and $\si(1)\neq 2$ in both. So we still don't know
which monomial is largest. We have compared the $1+6$ biggest components of the total degree,
namely the height and the $6$ exponents of the negative root vectors $F_{\be}$.

Thus we turn to comparing the remaining $6$ exponents of the positive root vectors $E_\be$.
Now care must be taken since, by \eqref{eq:totalDegree}, these are ordered in reverse
relative to the positive roots themselves.
Therefore, the next component to compare is the exponent of $E_{\be_{12}}$ because $\be_{12}$
is the smallest root. By \eqref{eq:identification}, this is the same as the
exponent of $t_{12}^{(1)}$ so we check if the permutations satisfy $\si(2)=1$. None of them do, so
we move on, checking $E_{\be_{13}}$ which amounts to checking if $\si(3)=1$.
Here we finally get a discrepancy, $(13)(24)$ satisfies this, but $(1324)$ does not.
Therefore $(13)(24)$ is the permutation that gives the leading term in $d_{42}$.

Of course, $(13)(24)=(1234)^2$, so this proves Theorem \ref{thm:leadingTerm} in the case $(r,s)=(4,2)$.
\end{example}

We will need several lemmas. The following notation will be used for a permutation $\si\in S_r$:
\[\mathrm{EX}(\si)=\{i\in\iv{1}{r}\mid \si(i)>i \},
\qquad 
\mathrm{AX}(\si)=\{i\in\iv{1}{r}\mid \si(i)<i \}.\]
Elements of $\mathrm{EX}(\sigma)$ (respectively $\mathrm{AX}(\sigma)$ are called \emph{exceedances} (respectively \emph{anti-exceedances}) for $\sigma$.

The following lemma describes which nonzero terms appear in $d_{rs}$.
\begin{lem}\label{lem:nonzero}
Let $s\in\iv{1}{r}$ and let $\si\in S_r$. Then the following two statements are equivalent.
\begin{enumerate}[{\rm (i)}]
\item $t_{\si(1)1}^{(k_1)}t_{\si(2)2}^{(k_2)}\cdots t_{\si(r)r}^{(k_r)}\neq 0$
for some $k\in\{0,1\}^r$ with $\sum_{i=1}^r k_i=s$;
\item $\#\mathrm{AX}(\si)\le s$ and $\#\mathrm{EX}(\si)\le r-s$.
\end{enumerate}
\end{lem}
\begin{proof}
This follows from the fact that $t_{ij}^{(1)}\neq 0$ iff $i\le j$ and $t_{ij}^{(0)}\neq 0$ iff $i\ge j$.
\end{proof}

Define the \emph{height} of a permutation $\si\in S_r$ by
\begin{equation}\label{eq:heightDef}
\height(\si):=\sum_{i=1}^r |\si(i)-i|.
\end{equation}
The motivation for this terminology comes from the fact that
\begin{equation}\label{eq:heightCorr}
\height(\si) = \height( t_{\si(1)1}^{(k_1)}t_{\si(2)2}^{(k_2)}\cdots t_{\si(r)r}^{(k_r)})
\end{equation}
where the right hand side is given by \eqref{eq:heightMonDef} and the identification \eqref{eq:identification}.

As the next step towards proving Theorem \ref{thm:leadingTerm}, we show that the permutation $\si$ which gives
the leading term of $d_{rs}$ has to be a derangement (i.e. $\si(i)\neq i\;\forall i\in\iv{1}{r}$).

\begin{lem}\label{lem:derangement}
Let $s\in\iv{1}{r}$ and let $\si\in S_r$ be a permutation such that
\[t_{\si(1)1}^{(k_1)}t_{\si(2)2}^{(k_2)}\cdots t_{\si(r)r}^{(k_r)}\neq 0\]
for some $k\in\{0,1\}^r$ with $\sum_i k_i=s$. Then there exists a $\widetilde{\si}\in S_r$ such that
\begin{enumerate}[{\rm (i)}]
\item \label{it:der1} $t_{\widetilde{\si}(1)1}^{(l_1)}\cdots t_{\widetilde{\si}(r)r}^{(l_r)}\neq 0$ for some
$l\in\{0,1\}^r$ with $\sum_i l_i=s$;
\item \label{it:der2} $t_{\widetilde{\si}(1)1}^{(l_1)}\cdots t_{\widetilde{\si}(r)r}^{(l_r)} \ge t_{\si(1)1}^{(k_1)}\cdots t_{\si(r)r}^{(k_r)}$ with respect to the De Concini-Kac filtration;
\item \label{it:der3} $\widetilde{\si}$ is a derangement.
\end{enumerate}
In particular, the permutation $\si$ such that \eqref{eq:ltdrs} holds is a derangement.
\end{lem}
\begin{proof}
If $\si$ already is a derangement, there is nothing to prove (take $\widetilde{\si}=\si$).
So suppose $f:=\#\{i\in S_r\mid \si(i)=i\}>0$. It is enough to construct $\widetilde{\si}$ satisfying
properties \eqref{it:der1}-\eqref{it:der2} with $\#\{i\in S_r\mid \widetilde{\si}(i)=i\} = f-1$
because then we can iterate this construction to arrive at a permutation satisfying all three conditions
\eqref{it:der1}-\eqref{it:der3}.

For brevity, we call $(i,\si(i))\in\iv{1}{r}^2$ a \emph{$\sigma$-jump} (respectively
\emph{$\sigma$-drop}) if $i$ is an exceedance (respectively anti-exceedance) for $\sigma$.
It will be useful to visualize a sequence $(i,\si(i),\ldots,\si^k(i))$ as a graph with vertex set $\{(x,\si^x(i))\mid x\in\iv{0}{k}\}\subset \Z^2$,
 connecting adjacent vertices $(a,b)$ and $(a+1,\sigma(b))$, as in Figure \ref{fig:example}. Then drops and jumps are simply as in Figure \ref{fig:dropjump}.

\begin{figure}[h]
\centering
\begin{tikzpicture}
  \tikzstyle{every node}=[draw,circle,fill=black,minimum size=3pt,inner sep=0pt]
  \draw (0,0) node [label=left:$1$] { } --
               (1,2) node [label=left:$3$] {} --
               (2,3) node [label=left:$4$] {} --
               (3,1) node [label=left:$2$] {} --
               (4,0) node [label=left:$1$] {};
\end{tikzpicture}
\caption{Pictorial representation of the cyclic permutation $(1342)$.}
\label{fig:example}
\end{figure}
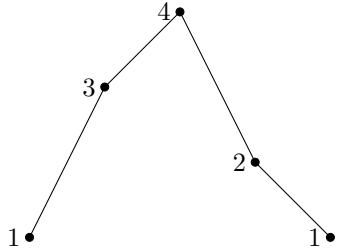

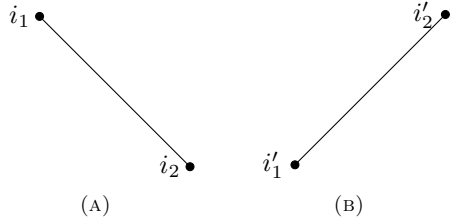
\begin{figure}[h]
\centering
\subfloat[]{
\begin{tikzpicture}
  \tikzstyle{every node}=[draw,circle,fill=black,minimum size=3pt,inner sep=0pt]
  \draw  (0,2) node [label=left: $i_1$] {} --
               (2,0) node [label=left: $i_2$] {};
\end{tikzpicture}}\qquad
\subfloat[]{
\begin{tikzpicture}
  \tikzstyle{every node}=[draw,circle,fill=black,minimum size=3pt,inner sep=0pt]
  \draw  (0,0) node [label=left: $i_1'$] {} --
               (2,2) node [label=left: $i_2'$] {};
\end{tikzpicture}}
\caption{A $\si$-drop (A) and a $\si$-jump (B). The diagrams mean $i_2=\si(i_1)$, $i_1>i_2$ and $i_2'=\si(i_1')$, $i_1'<i_2'$.}
\label{fig:dropjump}
\end{figure}

A $\sigma$-drop $(i_1,i_2)$ will be called \emph{drop-admissible} if we can ``add another drop between $i_1$ and $i_2$'',
that is, if there exists $j\in\iv{1}{r}$ with $\si(j)=j$ and $i_2<j<i_1$.
Then we can put $\widetilde{\si}=\si\circ (i_1\; j)$. With this $\widetilde{\si}$ we have
\[\#\mathrm{AX}(\widetilde{\si})=1+\#\mathrm{AX}(\si),\qquad \#\mathrm{EX}(\widetilde{\si})=\#\mathrm{EX}(\si).\]

Similarly, a $\sigma$-drop $(i_1,i_2)$ is \emph{jump-admissible} if there exists $j\in\iv{1}{r}$ with $\si(j)=j$ and $j\notin\iv{i_2}{i_1}$.
Then $\widetilde{\si}=\si\circ (i_1\; j)$ satisfies
\[\#\mathrm{AX}(\widetilde{\si})=\#\mathrm{AX}(\si),\qquad \#\mathrm{EX}(\widetilde{\si})=1+\#\mathrm{EX}(\si).\]
See Figure \ref{fig:dropscenarios} for an illustration of the possible scenarios in the case of a $\sigma$-drop.

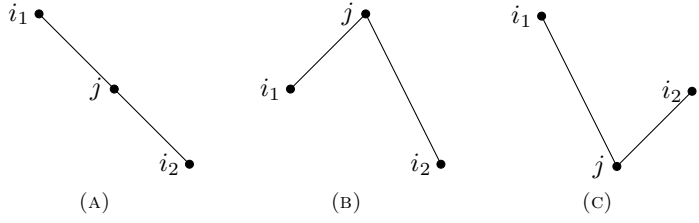
\begin{figure}[h]
\centering
\subfloat[]{
\begin{tikzpicture}
  \tikzstyle{every node}=[draw,circle,fill=black,minimum size=3pt,inner sep=0pt]
  \draw (0,2) node [label=left: $i_1$] {} --
               (1,1) node [label=left: $j$] {} --
               (2,0) node [label=left: $i_2$] {};
\end{tikzpicture}}\qquad
\subfloat[]{
\begin{tikzpicture}
  \tikzstyle{every node}=[draw,circle,fill=black,minimum size=3pt,inner sep=0pt]
  \draw  (3,1) node [label=left: $i_1$] {} --
               (4,2) node [label=left: $j$] {} --
               (5,0) node [label=left: $i_2$] {};
\end{tikzpicture}}\qquad
\subfloat[]{
\begin{tikzpicture}
  \tikzstyle{every node}=[draw,circle,fill=black,minimum size=3pt,inner sep=0pt]
  \draw   (6,2) node [label=left: $i_1$] {} --
               (7,0) node [label=left: $j$] {} --
               (8,1) node [label=left: $i_2$] {};
\end{tikzpicture}}
\caption{The three possible ways the $i_1,j,i_2$ piece of $\widetilde{\si}=\si\circ (i_1\; j)$ can look like, when $(i_1,i_2)$ is a $\si$-drop:
 $i_1<j<i_2$ (A),  $j>i_1,i_2$ (B), and $j<i_1,i_2$ (C). The $\si$-drop $(i_1,i_2)$ is drop-admissible in case (A), and jump-admissible in (B) and (C).}
\label{fig:dropscenarios}
\end{figure}

Analogously, a $\sigma$-jump $(i_1,i_2)$ is \emph{jump-admissible} if $\exists j\in\iv{1}{r}$ with $\si(j)=j$ and $i_1<j<i_2$.
A $\sigma$-jump $(i_1,i_2)$ is \emph{drop-admissible} if $\exists j\in\iv{1}{r}$ with $\si(j)=j$ and $j\notin \iv{i_1}{i_2}$.

We will now show that there always exists a jump-admissible $\si$-drop or $\si$-jump.

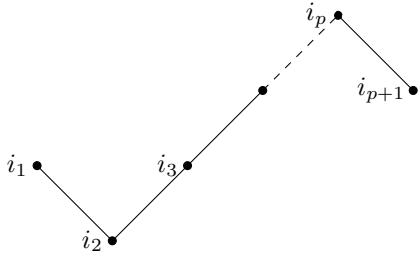
\begin{figure}[h]
\centering
\begin{tikzpicture}
  \tikzstyle{every node}=[draw,circle,fill=black,minimum size=3pt,inner sep=0pt]
  \draw (0,1) node [label=left: $i_1$] {} --
               (1,0) node [label=left: $i_2$] {} --
               (2,1) node [label=left: $i_3$] {} --
               (3,2) node [label=left: $ $] {};
  \draw[dashed](3,2)--(4,3) node [label=left: $i_p$] {};
  \draw (4,3)--(5,2) node [label=left: $i_{p+1}$] {};
\end{tikzpicture}
\caption{Illustration of a permutation $\si$ satisfying conditions (a)-(d).}
\label{fig:long}
\end{figure}

We know that $\si$ is not the identity permutation since $\sum_i k_i=s\ge 1$. Thus there exists a tuple
$(i_1,i_2,\ldots,i_p,i_{p+1})\in\iv{1}{r}^{p+1}$, where $p>2$, such that (see Figure \ref{fig:long})
\begin{enumerate}[{\rm (a)}]
\item $i_{j+1}=\si(i_j)$ for $j\in\iv{1}{p}$;
\item $i_1>i_2$;
\item $i_j<i_{j+1}$ for $j\in\iv{2}{p-1}$;
\item $i_p>i_{p+1}$.
\end{enumerate}
Note that we do not exclude the possibility that $(i_p,i_{p+1})=(i_1,i_2)$.
Also, since $\si$ is not a derangement, there is some $j\in\iv{1}{r}\setminus\{i_1,\ldots,i_{p+1}\}$ fixed by $\si$.

If $j\notin \iv{i_2}{i_1}$, then $(i_1,i_2)$ is a jump-admissible $\si$-drop (as in case (B) or (C) in Figure \ref{fig:dropscenarios}).
So suppose $i_1>j>i_2$. If $j<i_p$ then $(i_a,i_{a+1})$ is a jump-admissible $\si$-jump
for the $a\in\iv{2}{p-1}$ with $i_a<p<i_{a+1}$.
So suppose $j>i_p$. Then $(i_p,i_{p+1})$ is a jump-admissible $\si$-drop. This
proves that, provided $\si(j)=j$ for some $j$, there always exists a jump-admissible
$\si$-drop or $\si$-jump.

Similarly one proves there always exists a drop-admissible $\si$-drop or $\si$-jump.

If $\#\mathrm{AX}(\si)<s$ then we add a drop by putting $\widetilde{\si}=\si\circ (i\; j)$
where $(i,\si(i))$ is a drop-admissible $\si$-drop or $\si$-jump.
Then $\widetilde{\si}$ will have one more drop than $\si$ but the same number of jumps.
 That is, $\#\mathrm{AX}(\widetilde{\si})=1+\#\mathrm{AX}(\si)+1\le s$ and $\#\mathrm{EX}(\widetilde{\si})=\#\mathrm{EX}(\si)\le r-s$
which by Lemma \ref{lem:nonzero} ensures that property \eqref{it:der1} is satisfied.

Analogously, if instead $\#\mathrm{EX}(\si)<r-s$ we add a jump by putting $\widetilde{\si}=\si\circ (i\; j)$
for appropriate $i$.

Clearly $\widetilde{\si}$ has one less fixpoint than $\si$.

It remains to verify that property \eqref{it:der2} holds. The change from $\si$ to $\widetilde{\si}$
has the following effect on monomials:
\[t_{jj}^{(k_j)} t_{\si(i)i}^{(k_i)}\longmapsto t_{\widetilde{\si}(j)j}^{(k_j)} t_{\widetilde{\si}(i)i}^{(k_i)} = t_{\si(i)j}^{(k_j)} t_{ji}^{(k_i)}\]
(unchanged factors omitted).

If $j$ is not between $i$ and $\si(i)$, then by definition of the height \eqref{eq:heightDef} one checks that
 $\height(\widetilde{\si}) > \height(\si)$
so \eqref{it:der2} holds by just looking at the height, which is the most significant part of
the total degree (see \eqref{eq:totalDegree}).

If $j$ is between $i$ and $\si(i)$, then $\height(\widetilde{\si})=\height(\si)$
so we must compare roots in order to establish property \eqref{it:der2}.

Suppose $i<j<\si(i)$. Then the change from $\si$ to $\widetilde{\si}$ corresponds to
\[ t_{\si(i)i}^{(0)}t_{jj}^{(k_j)}\longmapsto t_{\si(i)j}^{(0)} t_{ji}^{(0)}\]
The change in total degrees is
\[ d(F_{\be_{i,\si(i)}}) \longmapsto d(F_{\be_{j,\si(i)}} F_{\be_{ij}})\]
Since $\be_{j,\si(i)}>\be_{i,\si(i)}, \be_{i,j}$ (recall the
ordering \eqref{eq:posRoots}) it follows that property
\eqref{it:der2} holds in this case. The case $i>j>\si(i)$ is
analogous, keeping in mind that $E_\be$ are ordered in reverse.
This finishes the proof of Lemma \ref{lem:derangement}.
\end{proof}

The following result describes the height of the permutation giving rise to the leading term.
\begin{lem}\label{lem:height}
Fix $r\in\Z_{>0}$ and let $s\in\iv{1}{r}$. Let $\si\in S_r$ be the permutation which gives rise to the
leading term of $d_{rs}$. That is,
\begin{equation}\label{eq:leading}
\lt(d_{rs}) = \la t_{\si(1)1}^{(k_1)}t_{\si(2)2}^{(k_2)}\cdots t_{\si(r)r}^{(k_r)}
\end{equation}
for some nonzero $\la\in\C$ and some $k\in\{0,1\}^r$ with $\sum_i k_i=s$.
Then
\begin{equation}\label{eq:height}
\height(\si) = 2s(r-s).
\end{equation}
\end{lem}
\begin{proof}
First we prove that $\height(\si)\ge 2s(r-s)$.
Let $\tau = (1\; 2\; \cdots\; r)^s$. We show that $\height(\tau)=2s(r-s)$.
Since
\[\tau(i)=\begin{cases}i+s,& i+s\le r\\ i+s-r,&i+s>r\end{cases}\]
we have by definition of $\height(\tau)$
\[\height(\tau) = \sum_{i=1}^{r-s} (i+s-i) + \sum_{i=r-s+1}^r (i-(i+s-r)) = 2s(r-s).\]
Since \eqref{eq:leading} is the leading term of $d_{rs}$,
we in particular have $\height(\si)\ge \height(\tau)=2s(r-s)$
by definition of total degree of a monomial \eqref{eq:totalDegree}.

It remains to show that $\height(\si)\le 2s(r-s)$. By Lemma \ref{lem:derangement},
$\si$ is a derangement. Thus
\[\height(\si)=\sum_{i=1}^r |\si(i)-i| = \sum_{i:\, \si(i)<i}(i-\si(i)) + \sum_{i:\,\si(i)>i} (\si(i)-i),\]
where the first sum has $s$ terms and the second has $r-s$ terms. Clearly we have the estimate
\begin{multline*}
\sum_{i:\, \si(i)<i}(i-\si(i)) + \sum_{i:\,\si(i)>i} (\si(i)-i) \\ \le
(r+(r-1)+\cdots + (r-s+1))-(1+2+\cdots+s) \\
+ (r+ (r-1)+\cdots (s+1)) - (1+2+\cdots +(r-s)) = 2s(r-s).
\end{multline*}
This proves the claim.
\end{proof}

We are now ready to prove Theorem \ref{thm:leadingTerm}.

\begin{proof}[Proof of Theorem \ref{thm:leadingTerm}]
The case $r=s$ is trivial: By \eqref{eq:drs}, $d_{rr}=\la\cdot t_{11}^{(1)}\cdots t_{rr}^{(1)}$, where $\la\in\C^\times$. Thus $d_{rr}$ has only one term,
corresponding to the identity permutation $(1)$. Thus the conjecture holds in this case because $(1\; 2\; \cdots \; r)^r = (1)$.
So we may assume $s<r$.

Let $\si\in S_r$ be the permutation which gives rise to the
leading term of $d_{rs}$. That is,
\[\lt(d_{rs}) = \la t_{\si(1)1}^{(k_1)}t_{\si(2)2}^{(k_2)}\cdots t_{\si(r)r}^{(k_r)}\]
for some nonzero $\la\in\C$ and some $k\in\{0,1\}^r$ with $\sum_i k_i=s$.
By Lemma \ref{lem:derangement}, $\si$ is a derangement. In particular, $k$ is uniquely determined:
$k_i=0$ iff $\si(i)>i$ and $k_i=1$ iff $\si(i)<i$. Moreover, since $\si$ is a derangement, Lemma \ref{lem:nonzero} implies that $s$ equals the number of anti-exceedances for $\sigma$:
\begin{equation}\label{eq:s}
s=\#\{i\in\iv{1}{r}\mid \si(i)<i\}.
\end{equation}
We will now show that
\begin{equation}\label{eq:perm_rs}
\si^{-1}(r)=r-s.
\end{equation}
This is equivalent to that $t_{r,r-s}^{(0)}$ occurs in $\lt(d_{rs})$.
By \eqref{eq:identification} and that the $K_i$ don't contribute to the total
degree, we have $d(t_{r,r-s}^{(0)})=d(F_{\be_{r-s,r}})$.
To show \eqref{eq:perm_rs}, note that
 $t_{r,r-s}^{(0)}$ occurs in the monomial corresponding to $\tau=(1\; 2\;\cdots\; r)^s$.
Thus it is enough to prove that if $t_{ji}^{(0)}$ occurs in the leading
monomial of $d_{rs}$ then $\be_{ij}\le \be_{r-s,r}$.

Suppose the opposite is true, i.e. that $\si^{-1}(j_0)=i_0\in \iv{r-s+1}{j_0-1}$ for some $j_0$ with $i_0<j_0\le r$.
We show that this leads to a contradiction in
the height of $\si$. We have
\begin{equation}\label{eq:sum0}
\height(\si)=\sum_{i=1}^r |\si(i)-i| = \sum_{i:\, \si(i)<i} (i-\si(i)) + \sum_{i:\, \si(i)>i} (\si(i)-i).
\end{equation}
The first sum has $s$ elements, by \eqref{eq:s}, and the second one has $r-s$ terms, since $\si$ is
a derangement. Since $\si(i_0)=j_0>i_0$, we may estimate the first sum from above by assuming that
$i$ runs through the $s$ largest elements of $\iv{1}{r}\setminus\{i_0\}$,
and $\si(i)$ just runs through the $s$ smallest elements of $\iv{1}{r}$. That is,
\begin{multline}\label{eq:sum1}
\sum_{i: \si(i)<i} (i-\si(i)) \le (r+(r-1)+\cdots + (r-s)-i_0) - (1+2+\cdots +s) \\= r-i_0 + s(r-s-1).
\end{multline}
On the other hand, $i_0$ does belong to the summation range of the other sum and therefore
\begin{multline}\label{eq:sum2}
\sum_{i:\, \si(i)>i} (\si(i)-i) \le (r+(r-1)+\cdots+(s+1)) - (1+2+\cdots + (r-s-1)+i_0) \\= (r-s-1)s+r-i_0,
\end{multline}
i.e. the sum of the $r-s$ largest elements of $\iv{1}{r}$ minus the smallest sum of $r-s$ elements
of $\iv{1}{r}$ requiring that one of them is $i_0$.
Combining \eqref{eq:sum0}-\eqref{eq:sum2} we obtain
\begin{equation}
\height(\si)\le 2(r-s-i_0) + 2s(r-s)<2s(r-s)
\end{equation}
since $i_0>r-s$ by assumption. This contradicts Lemma \ref{lem:height} and finishes the proof of \eqref{eq:perm_rs}.

Then, since $\be_{r-s-1,r-1}$ is the largest positive root of the form $\be_{r-s-1,j}$ where $j<r$,
$\be_{r-s-2,r-2}$ is the largest positive root of the form $\be_{r-s-2,j}$ with $j<r-1$, and so on,
we conclude that the leading term of $d_{rs}$ must have the form
\[\lambda\cdot t_{1+s,1}^{(0)}t_{2+s,2}^{(0)}\cdots t_{r,r-s}^{(0)} \cdot t_{\si(r-s+1),r-s+1}^{(k_1)}\cdots t_{\si(r)r}^{(k_s)}.\]
But $\sum k_i=s$ which forces $k_i=1$ for $i\in\iv{1}{s}$. So $\si(i)<i$ for $i\in\iv{r-s+1}{r}$.
Since $d(t_{ij}^{(1)}) = d(E_{\be_{ij}})$ for $i<j$ and by definition \eqref{eq:totalDegree} of the total degree,
the $E_{\be}$ are ordered in \emph{reverse} with respect to the order of the positive roots $\be$, we
are led to the question: What is the smallest possible root $\be_{ij}$ ($i<j$) which may still occur in the monomial?

We know that $\{\si(r-s+1),\si(r-s+2),\ldots,\si(r)\}=\{1,2,\ldots,s\}$. Thus, the smallest root we can get is
$\be_{1,r-s+1}$, obtained iff $\si(r-s+1)=1$. But this happens for the permutation $\tau=(1\; 2\;\cdots\; r)^s$.
So, to have any chance of getting a larger monomial we must continue. But at each step we see
that the smallest possible root is $\be_{i,r-s+i}$ for $i=1,2,\ldots,s$. This proves that
$(1\;2\;\cdots\;r)^s$ indeed is the permutation that gives the leading term of $d_{rs}$.
\end{proof}

For $a,b\in \iv{1}{N}$, $a\neq b$, and $u\in U_q$, let $\deg_{ab}(u)\in\Z_{\ge 0}$ denote the component of the De Concini-Kac filtration degree $d(u)\in(\Z_{\ge 0})^{2M+1}$ corresponding to the root $\beta_{ab}=\ep_a-\ep_b$.
The following result describes which positive roots that occur in the leading term of $d_{rs}$.
\begin{cor}\label{cor:root-degree}
If $1\le b<a\le N$ and $1\le s<r\le N$, then
\[\deg_{ba}(\lt(d_{rs}))=
 \begin{cases}
  1,& \text{$a-b=r-s$ and $a\le r$} \\
  0,& \text{otherwise}
 \end{cases}
\]
\end{cor}
\begin{proof}
By Theorem \ref{thm:leadingTerm}, $\lt(d_{rs})$ is (up to multiplication by some $K_i$ and a scalar) a product of distinct root vectors, and the positive root vector $E_\beta$, $\beta=\beta_{ba}$, occurs in $\lt(d_{rs})$ if and only if $(b,a)\in\{(1,r-s+1),(2,r-s+2),\ldots,(s,r)\}$ which is equivalent to $a-b=r-s$ and $a\le r$.
\end{proof}

\begin{cor}\label{cor:dab_lt} If $1\le b<a\le N$, $1\le s<r\le N$, then
\[\deg_{ba}\big(\lt(\prod_{1\le s<r\le N}d_{rs}^{k_{rs}})\big) = \sum_{r=a}^N k_{r,r-(a-b)}.\]
\end{cor}

\begin{rem}
Define
\begin{equation}
X(r,s) = t_{sr}^{(1)}
\end{equation}
for each $1\le s\le r\le N$. Then, by Theorem \ref{thm:leadingTerm}, $X(r,s)$ occurs in the leading term
of $d_{rs}$, however it sometimes \emph{does} occur in the leading term of some other $d_{ab}$, $(a,b)\neq (r,s)$.
Thus one cannot use the technique from \cite{FMO} to prove that $U_q(\mathfrak{gl}_n)$ is a Galois order. In fact we were not able to generalize our approach in any way to make it work.
\end{rem}


\end{document}